\documentclass[english, article]{amsart}

\usepackage{esint}
\usepackage[svgnames]{xcolor}
\usepackage[colorlinks,citecolor=red,pagebackref,hypertexnames=false,breaklinks]{hyperref}
\usepackage{pgf,tikz}
\usepackage{pdfsync}

\usepackage{subcaption}

\usepackage{dsfont}
\usepackage{url}
\usepackage[utf8]{inputenc}
\usepackage[T1]{fontenc}
\usepackage{lmodern}
\usepackage{babel}
\usepackage{mathtools}
\usepackage{amssymb}
\usepackage{graphicx}
\usepackage{lipsum}
\usepackage{mathrsfs}
\usepackage{color}

\newtheorem{theorem}{Theorem}[section]
\newtheorem{proposition}{Proposition}[section]

\numberwithin{equation}{section}

\title[Quantitative uniqueness of continuation result]{Quantitative uniqueness of continuation result related to Hopf's lemma}

\author[Mourad Choulli]{Mourad Choulli}
\address{Universit\'e de Lorraine, 34 cours L\'eopold, 54052 Nancy cedex, France}
\email{mourad.choulli@univ-lorraine.fr}

\author[Faouzi Triki]{Faouzi Triki}
\address{Laboratoire Jean Kuntzmann,  UMR CNRS 5224, 
Universit\'e  Grenoble-Alpes, 700 Avenue Centrale,
38401 Saint-Martin-d'H\`eres, France}
\email{faouzi.triki@univ-grenoble-alpes.fr}

\author[Qi Xue]{Qi Xue}
\address{Laboratoire Jean Kuntzmann,  UMR CNRS 5224, 
Universit\'e  Grenoble-Alpes, 700 Avenue Centrale,
38401 Saint-Martin-d'H\`eres, France}
\email{Qi.Xue@univ-grenoble-alpes.fr}

\thanks{The authors are supported by the grant ANR-17-CE40-0029 of the French National Research Agency ANR (project MultiOnde). }

\date{}

\begin{document}

\begin{abstract}
The classical Hopf's lemma can be reformulated as uniqueness of continuation result. We aim in the present work to quantify this property. We show precisely that if a solution $u$ of a divergence form elliptic equation attains its maximum at a boundary point $x_0$ then both $L^1$-norms of $u-u(x_0)$ on the domain and on the boundary are bounded, up to a multiplicative constant,  by the exterior normal derivative at $x_0$.
\end{abstract}



\maketitle

\section{Introduction}\label{section1}

Let $\Omega$ be a $C^{1,1}$ bounded domain of $\mathbb{R}^n$ ($n\ge 2$) with boundary $\Gamma$. All functions we consider are assumed to be real valued.

Fix $\kappa >1$, $0<\beta <1$ and  let $\Sigma$ be the set of functions $\sigma =(\sigma^{ij})\in C^{1,\beta}(\overline{\Omega},\mathbb{R}^{n\times n})$ satisfying $\sigma^{ji}=\sigma^{ij}$, $1\le i,j\le n$,
\[
 \kappa^{-1}|\xi|^2\le \sigma \xi\cdot \xi  \quad \mbox{for each}\; \xi  \in \mathbb{R}^n\quad \mbox{and}\quad  \|\sigma \|_{C^{1,\beta}(\overline{\Omega},\mathbb{R}^{n\times n})}\le \kappa .
\]

We associate to any $\sigma \in \Sigma$  the operator $L_\sigma$ acting as follows
\[
L_\sigma u=-\mbox{div}(\sigma \nabla u),\quad u\in C^2(\Omega).
\]

Define
\[
\mathscr{S}=\{u\in C^2(\Omega)\cap C^1(\overline{\Omega}); \;  L_\sigma u=0\; \mbox{for some}\; \sigma \in \Sigma\}
\]
and set
\[
M(u)=\{ x\in \overline{\Omega};\; u(x)=\max_{\overline{\Omega}}u\},\quad u\in \mathscr{S}.
\]

When $x\in \Gamma$ we denote by $\nu(x)$ the unit normal vector to $\Gamma$ pointing outward $\Omega$.

Let $u\in \mathscr{S}$ so that $\nabla u\ne0$. According to the strong maximum principle $M(u)\subset \Gamma$ (e.g. \cite[Theorem 3.5, page 35]{GT}) and by Hopf's Lemma\footnote{Also called Hopf-Oleinik-Zaremba's lemma or boundary point lemma.}, $\partial_\nu u(x)=\nabla u(x)\cdot \nu (x)>0$ for any $x\in M(u)$ (e.g \cite[Lemma 3.4, page 34]{GT}).

The first result of this kind goes back to the pioneering paper by Zaremba \cite{Za} and generalized later independently  by Hopf \cite{Ho} and Oleinik \cite{Ol}. We refer to the nice historical survey in the introduction of \cite{AN} including results with less regularity on the domain and the coefficients of the operators under consideration.

Hopf's lemma can be rephrased as uniqueness of continuation result: let $u\in \mathscr{S}$ and $x_0\in M(u)$. If $\partial_\nu u(x_0)=0$ then $u$ is identically equal to $u(x_0)$.

The following theorem quantify this uniqueness of continuation property.

\begin{theorem}\label{theorem1}
For any $u\in \mathscr{S}$ and $x_0\in M(u)$ we have
\begin{align}
&\|u(x_0)-u\|_{L^1(\Omega )}\le C\|u(x_0)-u\|_{L^1(\Gamma )},\label{me.0}
\\
&\|u(x_0)-u\|_{L^1(\Omega )}+\|u(x_0)-u\|_{L^1(\Gamma )}\le C\partial_\nu u(x_0),\label{me}
\end{align}
where $C=C(n,\Omega ,\kappa ,\beta)>0$ is a generic constant.
\end{theorem}

 We used $C^{1,\beta}$-regularity of the coefficients of the operator $L_\sigma$ only in the two-sided inequality \eqref{ke1} (which is contained in \cite[main Theorem in page 105]{HS}). For all other results $C^{0,1}$-regularity of the coefficients of the operator $L_\sigma$ is sufficient. It is not known presently  whether the result of \cite[main Theorem in page 105]{HS} can be extended to coefficients with $C^{0,1}$-regularity.

To prove Theorem \ref{theorem1} we modify the proof of Hopf's lemma itself, we make use the integral representation 
\[
u(x)-u(x_0)=\int_\Gamma K_\sigma (x,y)(u(x)-u(x_0))dS(y),\quad x\in \Omega,
\]
where $K_\sigma$ is Poisson type kernel associated to the operator $L_\sigma$. The proof is completed by showing beforehand two-sided inequality for $K_\sigma$ involving the weakly singular kernel $\mathrm{dist}(x,\Gamma)|x-y|^{-n}$.

Theorem \ref{theorem1} confirms numerical testing we obtained before. The details of these numerical testing are given in Appendix \ref{appendixB}. 

The rest of this text is organized as follows. In Section \ref{section2} we collect some properties of the Green function associated to the operator $L_\sigma$ and establish two-sided inequality for the Poisson type kernel $K_\sigma$. We give in Section \ref{section3} the proof of Theorem \ref{theorem1}. We also added two appendices. Appendix \ref{appendixA} contains the proof of a regularity result we used in Section \ref{section2}. While Appendix \ref{appendixB} is devoted to numerical testing.

\section{Preliminaries}\label{section2}

Define, where $\kappa >1$, $\Sigma_0$ as the set of functions $\sigma =(\sigma^{ij})\in C^{0,1}(\overline{\Omega},\mathbb{R}^{n\times n})$ satisfying $\sigma^{ji}=\sigma^{ij}$, $1\le i,j\le n$,
\[
 \kappa^{-1}|\xi|^2\le \sigma \xi\cdot \xi  \quad \mbox{for each}\; \xi  \in \mathbb{R}^n\quad \mbox{and}\quad  \|\sigma \|_{C^{0,1}(\overline{\Omega},\mathbb{R}^{n\times n})}\le \kappa .
\]
It is worth noticing that according to Rademacher's theorem (e.g. \cite[Theorem 2 in page 81]{EG}) $C^{0,1}(\overline{\Omega},\mathbb{R}^{n\times n})$ is continuously embedded in $W^{1,\infty}(\overline{\Omega},\mathbb{R}^{n\times n})$.

We associate to $\sigma \in \Sigma_0$ the symmetric bounded and coercive bilinear form
\[
\mathfrak{a}_\sigma (u,v)=(\sigma\nabla u|\nabla v)_2,\quad u,v\in H_0^1(\Omega),
\]
where $(\cdot |\cdot)_2$ is the usual scalar product  on $L^2(\Omega )$.

In this section we prove the following result, where
\[
\mathscr{U}=\{(x,y)\in \Omega \times \Omega;\; x\ne y\}.
\]

\begin{theorem}\label{theorem2}
Let $\sigma \in \Sigma_0$. Then
\\
(1) There exists a unique $G_\sigma\in L^1(\Omega\times \Omega )\cap C^1(\mathscr{U})$ satisfying
\[
\mathfrak{a}_\sigma (G_\sigma(\cdot ,y),v)=v(y),\quad v\in C_0^\infty(\Omega),\; y\in \Omega ,
\]
and $G_\sigma (x,y)=G_\sigma (y,x)$, $(x,y)\in \Omega \times \Omega$.
\\
(2) Let $\omega\Subset \omega_0\Subset \Omega$. Then $G_\sigma(x,\cdot )\in C^{1,\alpha}(\overline{\Omega}\setminus\omega_0)$ for each $x\in \Omega$, where $0\le \alpha =\alpha(n)<1$ is a constant. We have in addition 
\begin{equation}\label{G2.1}
\|G_\sigma(x,\cdot ) \|_{C^{1,\alpha}(\overline{\Omega}\setminus \omega_0)}\le c\quad \mbox{for each}\; x\in \omega
\end{equation}
and 
\begin{equation}\label{G2.2}
\|G_\sigma(x_1,\cdot )-G_\sigma(x_2,\cdot) \|_{C^{1,\alpha}(\overline{\Omega}\setminus \omega_0)}\le c|x_1-x_2|\quad \mbox{for each}\; x_1,x_2\in \omega ,
\end{equation}
where $c=c(n,\Omega ,\kappa,\omega,\omega_0)>0$ is a generic constant.
\end{theorem}

Prior to proving this theorem we state a regularity result of the solution of the Dirichlet BVP associated to $L_\sigma$.

Denote by $\gamma_0$ the bounded trace operator from $H^1(\Omega )$ onto $H^{1/2}(\Gamma )$ defined by 
\[
\gamma_0u =u_{|\Gamma},\quad u\in C^\infty (\overline{\Omega })
\]

Let $\sigma \in \Sigma_0$, $f\in L^\infty (\Omega )$ and consider the BVP
\begin{equation}\label{bvp1}
\left\{
\begin{array}{l}
-\mbox{div}(\sigma \nabla u)=f\quad \mbox{in}\; \Omega,
\\
\gamma_0u=0.
\end{array}
\right.
\end{equation}

Then Lax-Milgram's lemma allows us to conclude that the BVP \eqref{bvp1} has a unique variational solution $u=u_\sigma (f)\in H_0^1(\Omega)$:
\begin{equation}\label{1.0}
\mathfrak{a}_\sigma (u,v)=(f| v)_2\quad \mbox{for each}\; v\in H_0^1(\Omega ).
\end{equation}

\begin{theorem}\label{theorem1.1}
For any $\sigma \in \Sigma_0$ and $f\in L^\infty (\Omega )$, $u_\sigma (f)\in H^2(\Omega)\cap C^{1,\alpha}(\overline{\Omega})$, for some $0<\alpha =\alpha (n)<1$. Furthermore the following estimate holds
\begin{equation}\label{t1.1}
\|u_\sigma (f)\|_{H^2(\Omega )\cap C^{1,\alpha}(\overline{\Omega})}\le C\|f\|_{L^\infty (\Omega )},
\end{equation}
where $C=C(n,\Omega ,\kappa)>0$ is a constant.
\end{theorem}

We give the proof Theorem \ref{theorem1.1} in Appendix \ref{appendixA}.

\begin{proof}[Proof of Theorem \ref{theorem2}]
(1) is contained in \cite[Proposition 24 in page 625 and Proposition 26 in page 629]{DL}. 

(2) Fix $\omega \Subset \omega_0\Subset \Omega$ and $\phi \in C_0^\infty (\mathbb{R}^n\setminus \omega)$ satisfying $\phi=1$ in a neighborhood of $\overline{\Omega}\setminus\omega_0$. Then it is not difficult to check that $\phi G_\sigma(x,\cdot )$, $x\in \overline{\omega}$, is the solution of the BVP \eqref{bvp1} with
\[
f=2\sigma \nabla \phi\cdot \nabla G_\sigma (x,\cdot )+\sigma G_\sigma(x,\cdot )\Delta \phi.
\]
In light of (1) $f\in C(\Omega)$ and according to \cite[Inequalities (i) and (iv) of Theorem 3.3 in page 305]{GW} we have $f\in L^\infty (\Omega)$. The expected result follows by applying Theorem \ref{theorem1.1}. Furthermore from \cite[Inequality (1.8) of Theorem 1.1 in page 305]{GW} and \eqref{t1.1} we get
\[
\|G_\sigma(x,\cdot ) \|_{C^{1,\alpha}(\overline{\Omega}\setminus \omega_0)}\le c,\quad x\in \overline{\omega}.
\]
Here and henceforward $c=c(n,\Omega ,\omega ,\omega_0,\kappa)$.

Similarly, simple calculations show that $\phi [G_\sigma(x_1\cdot )-G_\sigma(x_2,\cdot )]$, $x_1,x_2\in \overline{\omega}$, is the solution of the BVP \eqref{bvp1} when
\[
f=2\sigma \nabla \phi\cdot \nabla [G_\sigma(x_1,\cdot )-G_\sigma(x_2,\cdot )]+\sigma [G_\sigma(x_1,\cdot )-G_\sigma(x_2,\cdot )])\Delta \phi.
\]
In light of \cite[Inequality (vi) of Theorem 3.3 in page 333]{GW} we get by applying the mean value theorem that
\[
\|f\|_{L^\infty(\Omega)}\le c|x_1-x_2|.
\] 
This inequality together with Theorem \ref{theorem1.1} yield
\[
\|G_\sigma(x_1,\cdot )-G_\sigma(x_2,\cdot ) \|_{C^{1,\alpha}(\overline{\Omega}\setminus \omega_0)}\le c|x_1-x_2|,\quad x_1,x_2\in \overline{\omega}.
\]
The proof is then complete.
\end{proof}

 We denote hereafter by $G$ the function $G_\sigma$ when $\sigma$ is identically equal to $1$. As usual the Poisson kernel is given by
 \[
 K(x,y)=-\partial_{\nu (y)}G(x,y),\quad x\in \Omega,\; y\in \Gamma .
 \]
 Note that according to Theorem \ref{theorem2} $K\in C(\Omega \times \Gamma)$.
 
 From \cite[Lemma 1 in page 21]{Zh} we have the following two-sided inequality
 \begin{equation}\label{Pk}
\varkappa^{-1}\frac{\mathrm{dist}(x,\Gamma)}{|x-y|^n}\le K(x,y) \le \varkappa \frac{\mathrm{dist}(x,\Gamma)}{|x-y|^n},\quad x\in \Omega,\; y\in \Gamma .
 \end{equation}
 where $\varkappa=\varkappa (n,\Omega)>1$ is a constant.
 
 Note that the domain in \cite[Lemma 1 in page 21]{Zh} is assumed to be of classe $C^2$ in order to guarantee the uniform interior sphere property but we know that this property is in fact satisfied by $C^{1,1}$-domains (see  Section \ref{section3}).
 
In the sequel we use the following notation, where $\sigma \in \Sigma_0$,
\[
K_\sigma(x,y) =-\partial_{\nu(y)}G_\sigma(x,y),\quad x\in \Omega,\; y\in \Gamma .
\]
Here again $K_\sigma \in C(\Omega \times \Gamma)$ by Theorem \ref{theorem2}.

\begin{proposition}\label{proposition1}
For all $\sigma \in \Sigma$ we have
\begin{equation}\label{ke}
\varkappa^{-1}\frac{\mathrm{dist}(x,\Gamma)}{|x-y|^n}\le K_\sigma (x,y) \le \varkappa \frac{\mathrm{dist}(x,\Gamma)}{|x-y|^n},\quad x\in \Omega,\; y\in \Gamma .
 \end{equation}
 where $\varkappa=\varkappa (n,\Omega, \kappa)>1$ is a constant.
\end{proposition}

\begin{proof}
In this proof $\varkappa=\varkappa (n,\Omega, \kappa)>1$ is a generic constant. Let $\sigma \in \Sigma$. Then  the following two sided inequality is contained in \cite[main Theorem in page 105]{HS}
\begin{equation}\label{ke1}
-\varkappa^{-1}G(x,y)\le -G_\sigma (x,y)\le -\varkappa G(x,y),\quad x,y\in \Omega.
\end{equation}

Fix $x\in \Omega$ and $y_0\in \Gamma$. Then for sufficiently small $t$ we have from \eqref{ke1}
\begin{align*}
\varkappa^{-1}\frac{-G(x,y_0-t\nu (y_0))+G(x,y_0)}{t}&\le \frac{-G_\sigma (x,y_0-t\nu (y_0))+G_\sigma (x,y_0)}{t}
\\
&\qquad \le \varkappa\frac{-G(x,y_0-t\nu (y_0))+G(x,y_0)}{t},\quad t>0,
\end{align*}
where we used that $G(x,y_0)=G_\sigma (x,y_0)=0$. Passing to the limit when $t$ goes to zero (observe that $G(x,\cdot)$ and $G_\sigma(x,\cdot)$ are $C^1$ up to the boundary)  we find
\[
-\varkappa^{-1}\partial_{\nu(x)}G(x,y_0)\le -\partial_{\nu(x)}G_\sigma(x,y_0)\le - \varkappa\partial_{\nu(x)}G(x,y_0)
\]
or equivalently
\[
\varkappa^{-1}K(x,y_0)\le K_\sigma (x,y_0)\le \varkappa K(x,y_0).
\]
We obtain the expected inequality by using \eqref{Pk}.
\end{proof}

\section {Proof of Theorem \ref{theorem1}}\label{section3}

In the sequel we use the fact that $\Omega$, which is $C^{1,1}$, admits the uniform interior sphere condition (e.g \cite[Theorem 1.0.9, page 7]{Ba}). This means that there exists $r>0$ so that for any $x\in \Gamma$, there exists $\tilde{x}\in \Omega$ with the property that $B(\tilde{x},r)\subset \Omega$ and $\partial B(\tilde{x},r)\cap \Gamma =\{x\}$.

Let $u\in \mathscr{S}$ so that $\nabla u\ne 0$ and pick $x_0\in M(u)$. Note that if $u$ is constant then \eqref{me} holds obviously.

Take $\tilde{x}_0\in \Omega$ with the property that $B(\tilde{x}_0,r)\subset \Omega$ and $\partial B(\tilde{x}_0,r)\cap \Gamma =\{x_0\}$. As in the proof of Hopf's Lemma, we introduce the function, where $0<\rho<r$ is arbitrary fixed,
\[
v(x)=e^{-\lambda |x-\tilde{x}_0|^2}-e^{-\lambda r^2}\quad \mbox{in}\; \omega=\{x\in \mathbb{R}^n;\; \rho<|x-\tilde{x}_0|<r\},
\]
where the constant $\lambda >0$ is to determined hereafter.

Straightforward computations show 
\[
L_\sigma v\ge (4\lambda ^2\rho ^2\kappa ^{-1}-c\lambda )e^{-\lambda r^2}\quad \mbox{in}\; \omega ,
\]
where $c=c(\kappa,r)$ is a constant. We fix then $\lambda=\lambda(\kappa,\rho,r)$ sufficiently large in such a way that $L_\sigma v\ge 0$.

In light of the strong maximum principle $\max_{|x-\tilde{x}_0|=\rho}u<u(x_0)$. Let then
\[
\epsilon =\frac{u(x_0)-\max_{|x-\tilde{x}_0|=\rho}u}{e^{-\lambda \rho^2}-e^{-\lambda r^2}}
\]
and
\[
w(x)=u(x)-u(x_0)+\epsilon v(x)\quad x\in \omega .
\]
Our choice of $\epsilon$ guarantees that $w\le 0$ on $\partial \omega$. As $L_\sigma w=\epsilon L_\sigma v\ge 0$ in $\Omega$ we derive from the weak maximum principle that $w\le 0$ in $\overline{\omega}$ (e.g. \cite[Theorem 3.1, page 32]{GT}). But $w(x_0)=0$ which means that $w$ achieves its maximum at $x_0$. In consequence $\partial_\nu w(x_0)\ge 0$ and hence
\[
\partial_\nu u(x_0)\ge -\epsilon \partial_\nu v(x_0)=2r\epsilon \lambda e^{-\lambda r^2}.
\]
In particular we have
\begin{equation}\label{1}
\partial_\nu u(x_0)\ge \gamma \min_{|x-\tilde{x}_0|=\rho}(u(x_0)-u(x)), 
\end{equation}
where 
\[
\gamma =\frac{2r \lambda e^{-\lambda r^2}}{e^{-\lambda \rho^2}-e^{-\lambda r^2}}.
\]

On the other hand, according to \cite[formula (8.95) in page 628]{DL} we know that
\begin{equation}\label{3}
u(x_0)-u(x)=\int_\Gamma K_\sigma (x,y)(u(x_0)-u(y))dS(y),\quad x\in \Omega ,
\end{equation}
which in light of the lower bound in \eqref{ke} yields
\begin{equation}\label{4}
\min_{|x-\tilde{x}_0|=\rho}(u(x_0)-u(x))\ge C\int_\Gamma (u(x_0)-u(y))dS(y).
\end{equation}
Here and until the rest of this proof $C=C(n,\Omega ,\kappa)>0$ is a generic constant.

We obtain by putting together inequalities \eqref{1} and \eqref{4} 
\begin{equation}\label{5}
\int_\Gamma (u(x_0)-u(y))dS(y)\le C\partial_\nu u(x_0).
\end{equation}

Let $\mathbf{d}$ be the diameter of $\Omega$. Then
\[
\int_\Omega \frac{\mathrm{dist}(x,\Gamma)}{|x-y|^{n}}dx\le\int_\Omega \frac{dx}{|x-y|^{n-1}}\le  |\mathbb{S}^{n-1}|\mathbf{d}\quad \mbox{for all}\; y\in \Gamma .
\]
This and the upper bound in \eqref{ke} show that $K_\sigma \in L^\infty (\Gamma ,L^1(\Omega))$ with 
\begin{equation}\label{6}
\|K_\sigma\|_{L^\infty (\Gamma ,L^1(\Omega))}\le C.
\end{equation}

Using again \eqref{3} we get by applying Fubini's theorem
\[
\int_\Omega (u(x_0)-u(x))dx =\int_\Gamma (u(x_0)-u(y))dS(y) \int_\Omega K_\sigma (x,y)dx.
\]
Therefore \eqref{6} gives 
\[
\int_\Omega (u(x_0)-u(x))dx \le C \int_\Gamma (u(x_0)-u(y))dS(y),
\]
which combined with \eqref{5} implies
\[
\|u(x_0)-u\|_{L^1(\Omega )}\le C\partial_\nu u(x_0).
\]
Hence
\[
\|u(x_0)-u\|_{L^1(\Omega )}+\|u(x_0)-u\|_{L^1(\Gamma )}\le C\partial_\nu u(x_0),
\]
as expected.

\appendix

\section{Proof of Theorem \ref{theorem1.1}}\label{appendixA}

In all this proof $C=C(n,\Omega ,\kappa)>0$ is a generic constant.

We denote Poincar\'e's constant of $\Omega$  by $\mathfrak{p}^2$:
\[
\|w\|_{L^2(\Omega )}\le \mathfrak{p} \|\nabla w\|_{L^2(\Omega)}\quad \mbox{for each}\; w\in H_0^1(\Omega ).
\]

Let $\sigma\in \Sigma_0$, $f\in L^\infty (\Omega)$ and $u=u_\sigma (f)$. We get in a straightforward manner by taking $v=u$ in \eqref{1.0}
\[
\|\nabla u\|_{L^2(\Omega)}^2\le \kappa \|f\|_{L^2(\Omega)}\|u\|_{L^2(\Omega)}.
\]
This inequality together with Poincar\'e's inequality give
\begin{equation}\label{1.1}
\|\nabla u\|_{L^2(\Omega)}\le \mathfrak{p}\kappa \|f\|_{L^2(\Omega)}.
\end{equation}

We can then apply \cite[Theorem 8.53 in page 326]{RR} and its proof in order to conclude that $u\in H^2(\Omega )$ and
\begin{equation}\label{1.2}
\|u\|_{H^2(\Omega )}\le C\|f\|_{L^2(\Omega)}.
\end{equation}

We now discuss $C^{1,\alpha}$ regularity of $u$. To this purpose we shall use repeatedly \cite[Theorem 9.14 in page 240 and Theorem 9.15 in page 241]{GT} concerning $W^{2,p}$ elliptic regularity and the corresponding $W^{2,p}$ a priori estimate.

Let us then consider first the case $n=2$. In that case since $H^1(\Omega )$ is continuously embedded in 
$C(\overline{\Omega})$ we derive that 
\[
L_\sigma u=f\in L^p(\Omega ),\quad 1<p<\infty .
\]
Whence $u\in W^{2,p}(\Omega)$ and 
\[
\|u\|_{W^{2,p}(\Omega )}\le C\|f\|_{L^p(\Omega)}.
\]
This and \eqref{1.2} imply
\begin{equation}\label{1.4}
\|u\|_{W^{2,p}(\Omega )}\le C\|f\|_{L^p(\Omega )}.
\end{equation}

For $0<s <1$, we can apply the preceding result with $p_s=2/(1-s)$. Noting that $W^{2,p_s}(\Omega)$ is continuously embedded in $C^{1,s}(\overline{\Omega})$ we deduce that $u\in C^{1,s}(\overline{\Omega})$ and from \eqref{1.4} we have
\begin{equation}\label{1.5}
\|u\|_{C^{1,s}(\overline{\Omega})}\le C\|f\|_{L^\infty (\Omega )}.
\end{equation}

Similarly when $n=3$, using that $H^1(\Omega )$ is continuously embedded in $L^6(\Omega )$, we obtain that $u\in W^{2,6}(\Omega )$. But $W^{2,6}(\Omega )$ is continuously embedded in $C^{1,1/2}(\overline{\Omega} )$. Therefore $u\in C^{1,1/2}(\overline{\Omega} )$ and the following estimate holds
\begin{equation}\label{1.6}
\|u\|_{C^{1,1/2}(\overline{\Omega})}\le C\|f\|_{L^\infty(\Omega )}.
\end{equation}

Assume that $n\ge 4$ and let $k_n\ge 1$ be the smallest integer $k$ so that $k\ge n/2-1$. Define 
\[
p_1=\frac{2n}{n-2}\quad p_k=\frac{np_{k-1}}{n-p_{k-1}},\quad 2\le k\le k_n.
\]
It is then not hard to check that
\[
p_k=\frac{2n}{n-2k}\quad 0\le k\le k_n.
\]
 
We proceed as before. First we use that $H^1(\Omega)$ is continuously imbedded in $L^{p_1}(\Omega)$ to derive that $u\in W^{2,p_1}(\Omega)$ and
\[
\|u\|_{W^{2,p_1}(\Omega )}\le C\|f\|_{L^{p_1}(\Omega )}.
\]
We use then that $W^{1,p_1}(\Omega)$ is continuously embedded in $L^{p_2}(\Omega)$ and we repeat the preceding argument in order to obtain that $u\in W^{2,p_2}(\Omega)$ and
\[
\|u\|_{W^{2,p_2}(\Omega )}\le C\|f\|_{L^{p_2}(\Omega )}.
\]
By induction in $k$ we get at the end that $u\in W^{2,p_{k_n}}(\Omega )$ and 
\[
\|u\|_{W^{2,p_{k_n}}(\Omega )}\le C\|f\|_{L^{p_{k_n}}(\Omega )}.
\]

When $n=2m+1$, $m\ge 1$ then $k_n=m$. In that case as $W^{2,p_{k_n}}(\Omega )$ is continuously embedded in $C^{1,1/2}(\overline{\Omega})$ we obtain that $u\in C^{1,1/2}(\overline{\Omega})$ and
\begin{equation}\label{1.7}
\|u\|_{C^{1,1/2}(\overline{\Omega})}\le C\|f\|_{L^\infty(\Omega )}.
\end{equation}

If $n=2m$, $m\ge 2$, we have $p_{k_n}=n$. In particular $u\in W^{2,n-1/2}(\Omega )$ and
\[
\|u\|_{W^{2,n-1/2}(\Omega )}\le C\|f\|_{L^{n-1/2}(\Omega )}.
\]
Using that $W^{2,n-1/2}(\Omega )$ is continuously embedded in $L^{n(2n-1)}(\Omega)$ we get that $u\in W^{2,n(2n-1)}(\Omega )$ and
\[
\|u\|_{W^{2,n(2n-1)}(\Omega )}\le C\|f\|_{L^{n-1/2}(\Omega )}.
\]
We finally obtain $u\in C^{1,\alpha}(\overline{\Omega})$, with $\alpha =(2n-2)/(2n-1)$, and
\begin{equation}\label{1.8}
\|u\|_{C^{1,\alpha}(\overline{\Omega})}\le C\|f\|_{L^\infty(\Omega )}.
\end{equation}

\section{Numerical testing}\label{appendixB}

We limit our numerical testing to (sufficiently smooth) isotropic $\sigma$ with $n=2$, $\Omega = B(0,1)$ and the following sequence of boundary conditions:
\[
h_k (x_1,x_2) = \left(\frac{x_2+3}{4}\right)^{1/k}, \quad n\in \mathbb{N},\; \mbox{and}\; |(x_1,x_2)|=1.
\]
It is not hard to check that $(h_k)$ converges uniformly on $\mathbb{S}^1$ to the constant function equal to $1$.

Denote by $u_k$ the solution of the BVP
\[
\left\{
\begin{array}{l}
-\mbox{div}(\sigma \nabla u_k)=0\quad \mbox{in}\; \Omega,
\\
\gamma_0u_k=h_k.
\end{array}
\right.
\]
According to the maximum principle $u_k$  attains its maximum at $x_0=(0,1)$.

\begin{figure}[ptb]
\centering
\begin{subfigure}[p]{.24\textwidth}
    \centering
   	\includegraphics[scale=.23,trim={0cm 0cm 1cm 0cm},clip]{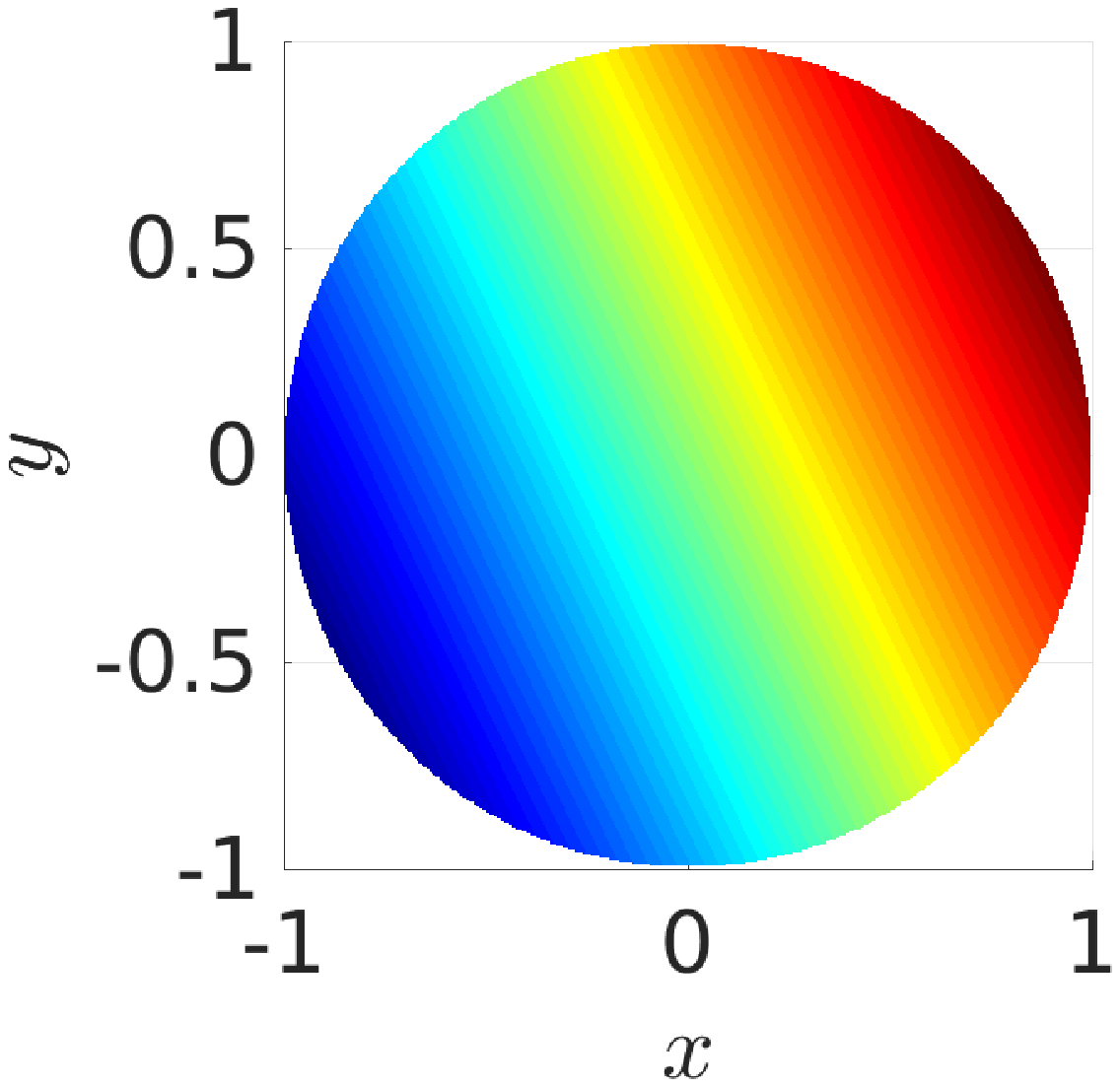}
   	\caption{Linear}
   	\label{fig:linear}
\end{subfigure}
\begin{subfigure}[p]{.24\textwidth}
    \centering
    \includegraphics[scale=.23,trim={0cm 0cm 1cm 0cm},clip]{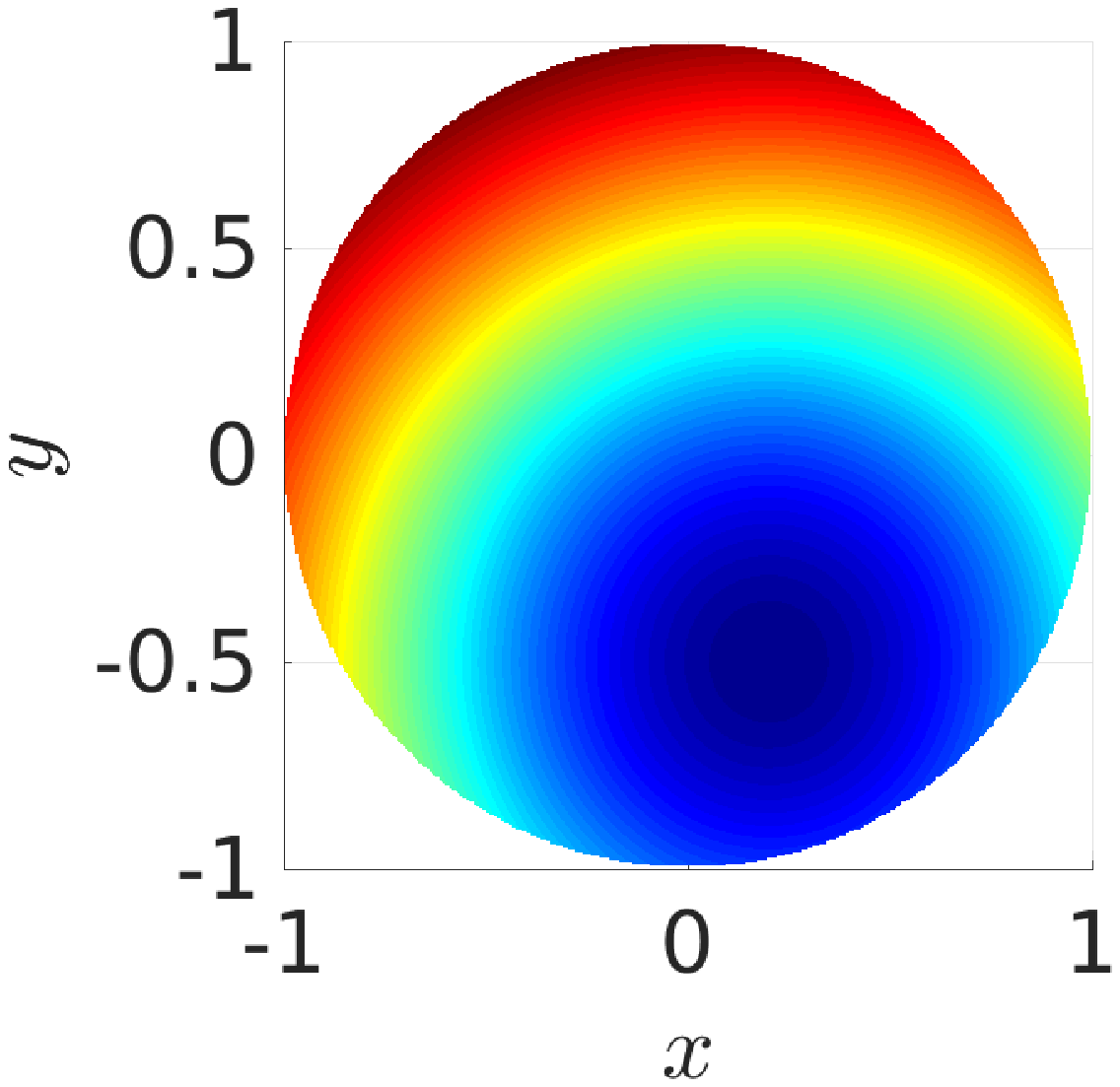}
    \caption{Gaussian}
   	\label{fig:Gaussian}
\end{subfigure}
\begin{subfigure}[p]{.24\textwidth}
    \centering
    \includegraphics[scale=.23,trim={0cm 0cm 1cm 0cm},clip]{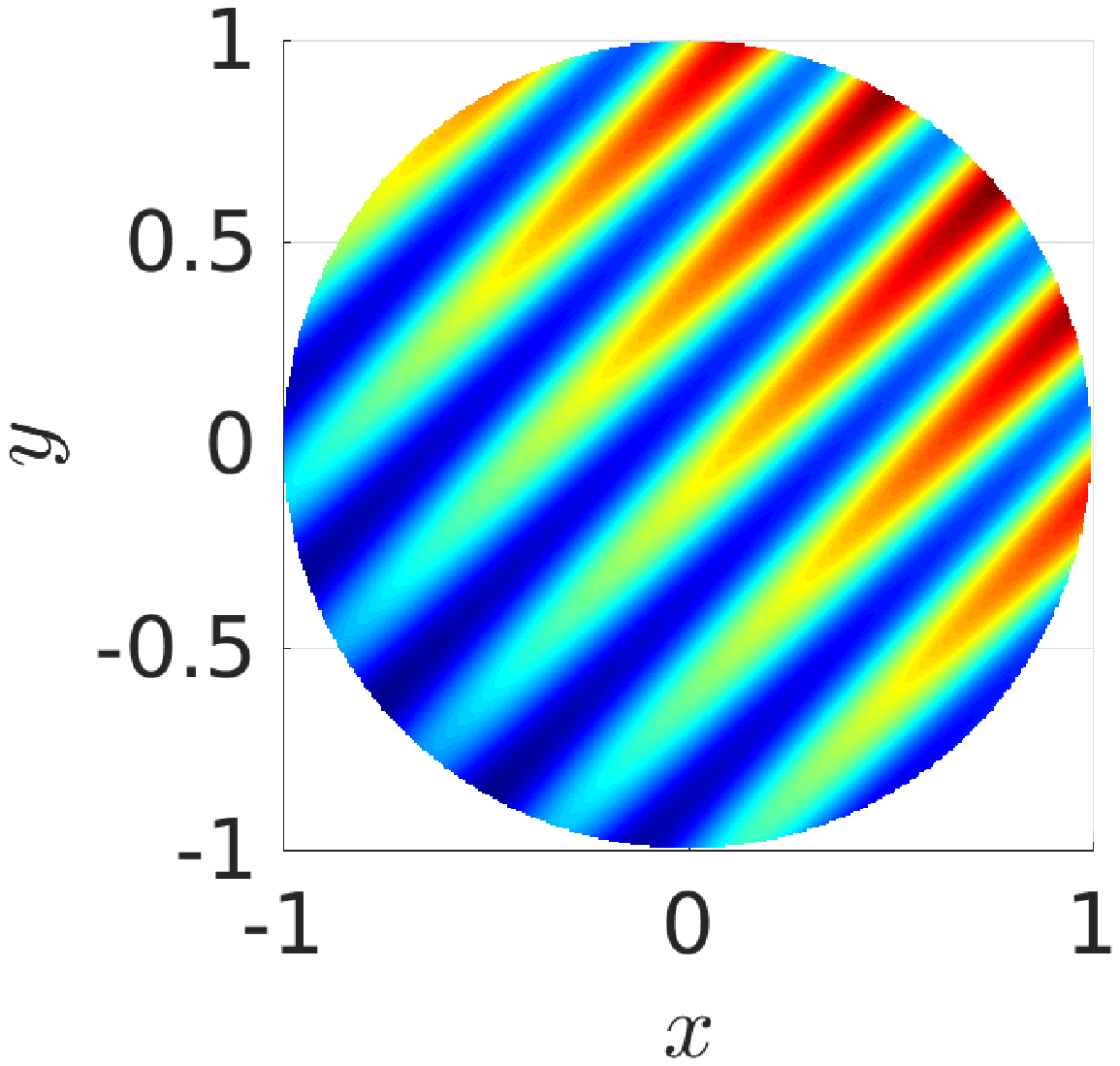}
	\caption{Oscillating}
   	\label{fig:oscillating}
\end{subfigure}
\begin{subfigure}[p]{.24\textwidth}
    \centering
    \includegraphics[scale=.23,trim={0cm 0cm 0cm 0cm},clip]{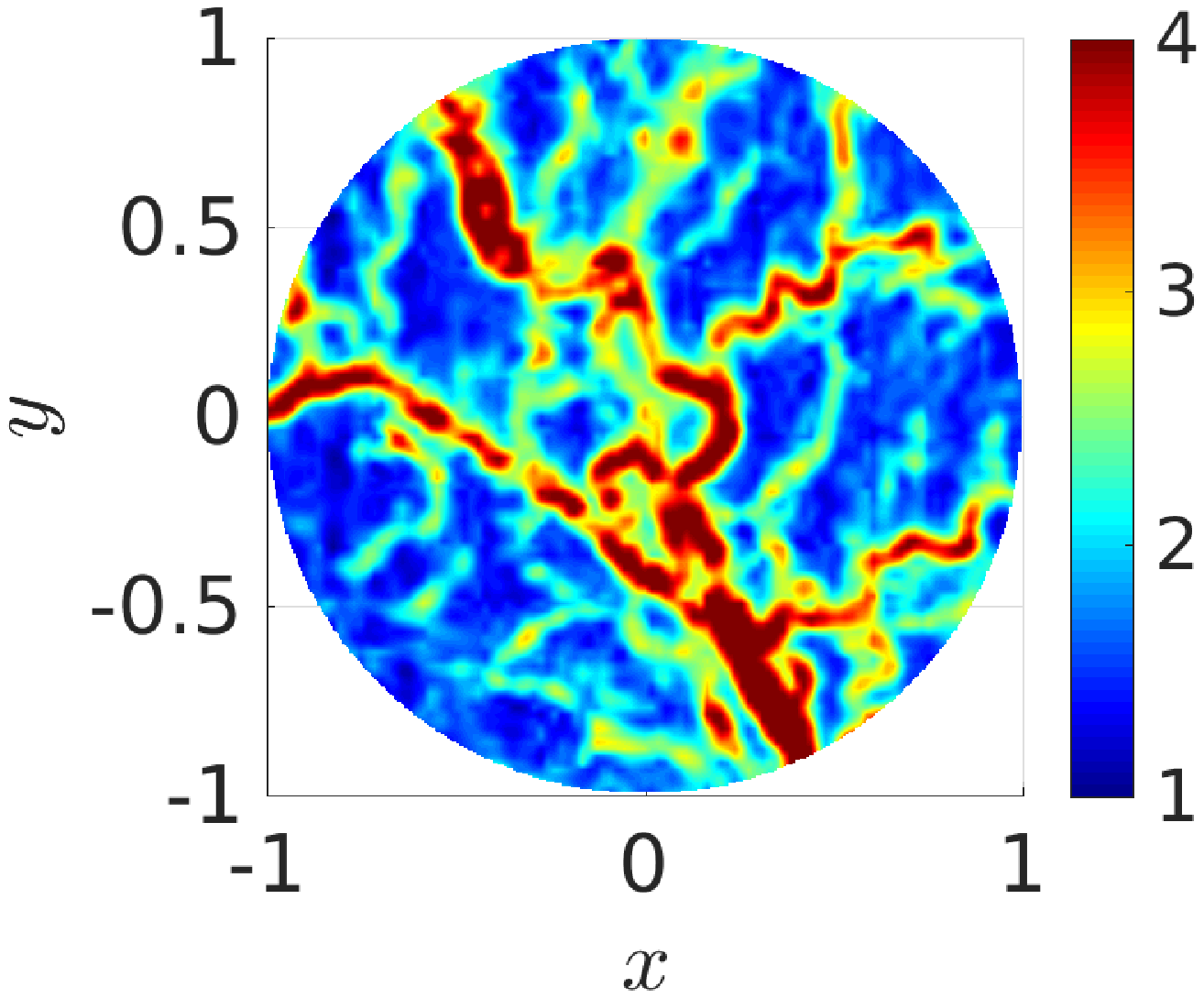}
    \caption{Realistic}
   	\label{fig:real}
\end{subfigure}
\caption{Different choices of the coefficient $\sigma$.}
\label{fig:sigma}
\end{figure}

We considered four different choices of the coefficient $\sigma$ in Fig.\ref{fig:sigma}: linear, Gaussian, oscillating and realistic.
We restrict the highest and lowest values to $4$ and $1$ respectively. 

\begin{figure}[ptb]
\centering
\includegraphics[scale=.6,trim={0cm 0cm 0cm 0cm},clip]{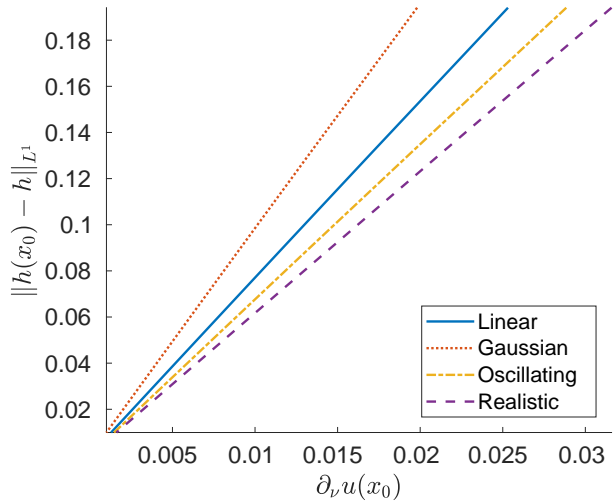}
\caption{Convergence of $\|h(x_0)-h\|_{L^1(\Gamma)}$ with respect to the normal derivative $\partial_\nu u(x_0)$.}
\label{fig:convergence}
\end{figure}

We observe in Fig.\ref{fig:convergence} that $\|h(x_0)-h\|_{L^1(\Gamma)}$ converges to $0$ linearly as $\partial_\nu u(x_0)\rightarrow 0$.

\end{document}